\documentclass{amsart}
\usepackage{amsmath}%
\usepackage{amsthm}
\usepackage{amsfonts}%
\usepackage{amssymb}%
\usepackage[abbrev, nobysame]{amsrefs}
\usepackage{enumitem}
\usepackage{graphicx}
\usepackage{fullpage}
\usepackage{color}
\usepackage[utf8]{inputenc}
\usepackage[english]{babel}
\usepackage{soul}
\usepackage{cancel}
\usepackage{tikz}
\usepackage{caption}

\newcommand{\Rb}{\mathbb{R}}

\newtheorem{theorem}{Theorem}[section]

\newtheorem{corollary}[theorem]{Corollary}

\newtheorem{lemma}[theorem]{Lemma}

\newtheorem{remark}[theorem]{Remark}

\begin{document}
\title{Capacity estimates and improved lower bounds for the inner radius of nodal domains}
\author{Philippe Charron}
\address{Department of Mathematics, University College London, Gower Street, London, WC1E 6BT, UK}
\email{p.charron@ucl.ac.uk}
\dedicatory{In loving memory of Vincent Charron}
\date{\today}

\begin{abstract}
    We show that for every closed, smooth manifold $(M,g)$ of dimension $d$, there exists $c(g)$ such that any nodal domain $\Omega_\lambda$ of a Laplace eigenfunction with eigenvalue $\lambda$ contains a geodesic ball of radius at least $c(g) \lambda^{-1/2} \log\log(\lambda)^{-1/2}$ if $d=3$ and $c(g) \lambda^{-1/2} \log(\lambda)^{-\frac{d-3}{2}}$ if $d >3$. This ball is centered at any point at which the eigenfunction attains its maximum in absolute value within the nodal domain. Furthermore, we show that for any $d \geq 3$, there exist sequences of $\lambda$-nodal domains on $\mathbb{T}^d$  whose inner radius is of order $o(\lambda^{-1/2})$.
\end{abstract}

\maketitle


\section{Introduction and main result}

Let $(M,g)$ be a smooth closed manifold of dimension $d$. Let $\Delta_g = -\text{div}_g(\nabla_g)$ be the positive-definite Laplace-Beltrami operator on $M$ and $\phi_\lambda$ be a solution of $\Delta_g \phi_\lambda = \lambda \phi_\lambda$. A nodal domain $\Omega_\lambda$ of $\phi_\lambda$ is a connected component of the set $\{ \phi_\lambda \neq 0 \}$. We are interested in the inner radius of nodal domains, i.e.  the radius of the largest geodesic ball that can be inscribed inside $\Omega_\lambda$ and which we will denote by $\text{Inrad}(\Omega_\lambda)$.

It is a classical result (see \cites{berard-meyer, bruning78}) that $\text{Inrad}(\Omega_\lambda) \leq C(g) \lambda^{-1/2}$. If $d=2$, it was proven in \cite{mang-inrad} that $\text{Inrad}(\Omega_\lambda) \geq c(g) \lambda^{-1/2}$. 

Various universal lower bounds in dimension $d \geq 3$ have been proven over the years, see \cites{geor, mang-inrad, mang-la} . The sharpest lower bound to date has been proven recently by Mangoubi and the author: if $x_0$ is any point where $|\phi_\lambda|$ attains its maximum on a nodal domain $\Omega_\lambda$, then there is a geodesic ball of radius $c(g) \lambda^{-1/2} (\log \lambda)^{\frac{2-d}{2}}$ centered at $x_0$ which is inscribed in $\Omega_\lambda$.

    The main result of this paper is an improved lower bound on the inner radius in dimensions $3$ and above:

\begin{theorem}\label{theorem:maintheorem}
    Let $(M,g)$ be a smooth closed manifold of dimension $d \geq 3$ and $\phi_\lambda$ be a solution of \linebreak $\Delta_g \phi_\lambda = \lambda \phi_\lambda$. There exists a constant $c(g)$ such that for any nodal domain $\Omega_\lambda$ of $\phi_\lambda$, 

    \begin{align}
        \text{Inrad}(\Omega_\lambda) &\geq  c\lambda^{-1/2}(\log\log\lambda)^{-1/2}  \quad & \text{if} \quad d=3 \, , \label{eq1}\\
    \text{Inrad}(\Omega_\lambda) &\geq c \lambda^{-1/2}(\log\lambda)^{\frac{3-d}{2}} \quad & \text{if} \quad d \geq 4 \, . \label{eq2}
    \end{align}

    Furthermore, that ball can be inscribed around any point where $|\phi_\lambda|$ attains its maximum in $\Omega_\lambda$.
\end{theorem}

A natural open question is: what is the optimal lower bound for the inner radius of nodal domains? 

If $M$ is a flat torus with edge lengths which are rationally independent, the spectrum is simple, nodal domains are hyper-rectangles and one can easily verify that there exists a constant $c$ which depends only on the edge lengths such that $\text{Inrad}(\Omega_\lambda) \geq c \lambda^{-1/2} $.

However, it is not true that a lower bound of the form $c(g) \lambda^{-1/2}$ holds in the general case:

\begin{theorem}\label{theorem:maintheoremshrinking}
    For any $d \geq 3$, there is a sequence $k(d)$, solutions to  $\Delta\phi_{\lambda_{k}} = \lambda_k \phi_{\lambda_k}$ on the flat square torus $\mathbb T^d$ and a nodal domain $\Omega_{\lambda_k}$ such that $$\text{Inrad}(\Omega_{\lambda_k}) =o\left(\lambda_k^{-1/2}\right) \, .$$
\end{theorem}
We will include a sketch of the proof of Theorem \ref{theorem:maintheoremshrinking}  in section \ref{section:proofofshrinking} as we could not find a complete proof in the literature. We note that most of the ideas of the proof can already be found in \cite{enc-persal1}, which uses results from \cite{enc-persal2} and \cite{enc-persal3}.

\subsection{Acknowledgments}
Work for this project started when the author was at Technion - Israel Institute of Technology, then at the University of Geneva and finally at University College London. The author acknowledges the support of the Zuckerman STEM Leadership Program, the SNSF and the Leverhulme Trust.

The author is very grateful to Dan Mangoubi for discussing various aspects of the problem as well as for reading an early version of this article and making useful suggestions, to Almut Burchard for  pointing out various results on rearrangements and to Iosif Polterovich for his continuons support throughout the project.

\section{Notation and tools}
\subsection{Notation}\label{subsection:notation}
\begin{itemize}
\item Throughout the paper, we will be working in dimension $d \geq 3$.
\item Constants $c$ and $C$ will change from line to line, but $c$ will refer to a lower bound and $C$ will refer to an upper bound. 
    \item Throughout the paper, we will work exclusively on charts.

\item For any $x \in M$, let $U_x \subset \mathbb R^d$ be the largest geodesic normal chart centered at $x$ which sends $x$ to the origin.

    \item Any ball $B=B_R(x)$  will refer to an open Euclidean ball. We define $2B=B_{2R}(x)$ as the concentric ball of twice the radius. We will also use the shorthand notation $B_R$ for $B_R(0)$. 
    The choice of Euclidean balls (rather than geodesic balls) is motivated by the use of rearrangements defined below which which have nice properties in the Euclidean setting (Theorem \ref{theorem:monotonicityrearrangements}).

\item Let $r_0$ be chosen such that for any $x \in M$, $B_{r_0}(0) \subset U_x$.

\item Let $f$ be a function defined in a ball $B_{2R}(y)$, . The doubling index of $f$ on $B_R(y)$ is defined as
$$\mathcal{N}(f, B_R(y)):=\log_2\frac{\sup_{B_{2R}(y)} |f|}{\sup_{B_{R}(y)} |f|} \, .$$

    \item Let $\mu_k$ be the standard $k$-dimensional Hausdorff measure for subsets of $\mathbb R^d$.

    \subsection{Rearrangements}\label{subsection:rearrangements}
We now introduce two types of rearrangements of sets, following the definitions in \cite{Sarvas}. We assume $\Omega \subset \mathbb R^d$ to be compact 

\item 
For $r \in \mathbb R^+$ and for $\delta \in [0,w_{d-1}r^{d-1}]$, where $w_{d-1}$ is the area of the unit sphere in $\mathbb R^d$, let \linebreak $Sc(r,\delta) \subset \partial B_r$ be the spherical cap centered around $(r,0,\ldots,0)$ such that  $\mu_{d-1}(Sc(r,\delta)) = \delta$.

For any $r >0$, let $\delta_\Omega(r) := \mu_{d-1}(\Omega \cap r \mathbb S^{d-1})$. 

The spherical cap rearrangement of $\Omega$, $Sc(\Omega)$, is defined as

$Sc(\Omega) := \bigcup\limits_{r >0} Sc(r,\delta_\Omega(r))$.

\item
Now, let $H$ be any hyperplane in $\mathbb R^d$. For any $x \in H$, let $l_x$ be the line passing through $x$ which is orthogonal to $H$ and $I_\Omega(x) \subset l_x$ the line segment of length $\mu_1(\Omega \cap l_x)$ centered at $x$. 

The Steiner rearrangement of $\Omega$ with respect to $H$, $St_H(\Omega)$, is defined as

$St_H(\Omega):= \bigcup\limits_{x \in H} I_\Omega(x)$.

An example of these two types of rearrangements is given in Figure \ref{picture:rearrangements}.

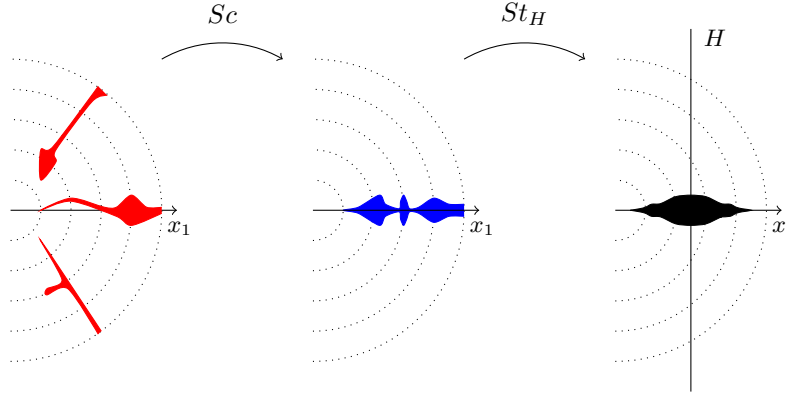
\begin{figure}
\centering

\begin{tikzpicture}[scale=0.4]


\draw[dotted] (0,-5) arc (-90:90:5);
\draw[dotted] (0,-4) arc (-90:90:4);
\draw[dotted] (0,-3) arc (-90:90:3);
\draw[dotted] (0,-2) arc (-90:90:2);
\draw[dotted] (0,-1) arc (-90:90:1);
\draw[->]  (5,5) arc (120:60:4);
\draw[]  (7,6.5) node{$Sc$};

\draw[->] (0,0) -- (5.5,0) ;
\draw[] (5.6,-0.6) node{ \small $x_1$};


\draw[red, fill] plot [smooth, tension=0.5] coordinates { (3.2,3.8) (2.7, 3.6) (1.5, 2) (1.5, 1.6) (1,1) (1, 2) (1.4, 2.1) (2.9,4.1)  };

\draw[red, fill] plot [smooth, tension=0.5] coordinates { (5, 0.1) (4.5, 0.1) (4,0.5) (3.5, 0.1) (3.1,0.1) (2,0.4) (1,0) (2, 0.3) (3.5,-0.2) (4,-0.5) (5, -0.1)};

\draw[red, fill] plot [smooth, tension=0.5] coordinates { (2.9,-4.1) (1.9, -2.6) (1.2, -2.8) (1.2, -2.6) (1.7,-2.3) (1,-1) (3,-4)  };


 \begin{scope}[shift={(10,0)}]

\draw[dotted] (0,-5) arc (-90:90:5);
\draw[dotted] (0,-4) arc (-90:90:4);
\draw[dotted] (0,-3) arc (-90:90:3);
\draw[dotted] (0,-2) arc (-90:90:2);
\draw[dotted] (0,-1) arc (-90:90:1);
\draw[->]  (5,5) arc (120:60:4);
\draw[]  (7,6.5) node{$St_H$};

\draw[blue, fill] plot [smooth, tension=0.5] coordinates { (5,0.2) (4.5, 0.2) (4,0.4) (3.5, 0.1) (3.2, 0.05) (3.1, 0.3) (3, 0.5) (2.9, 0.3) (2.85, 0.05) (2.4, 0.2)  (2.2, 0.5) (1.5, 0.1)  (1,0) (1.5, -0.1) (2.2,-0.5)  (2.4,-0.2) (2.85, -0.05) (2.9, -0.3) (3,-0.5) (3.1,-0.3) (3.2, -0.05) (3.5, -0.1) (4,-0.4) (4.5, -0.2)      (5,-0.2)   };

\draw[->] (0,0) -- (5.5,0) ;
\draw[] (5.6,-0.6) node{ \small $x_1$};

            \end{scope}

            \begin{scope}[shift={(20,0)}]

\draw[dotted] (0,-5) arc (-90:90:5);
\draw[dotted] (0,-4) arc (-90:90:4);
\draw[dotted] (0,-3) arc (-90:90:3);
\draw[dotted] (0,-2) arc (-90:90:2);
\draw[dotted] (0,-1) arc (-90:90:1);
\draw[] (2.5, 6) -- (2.5, -6);
\draw[black] (3.3,5.7) node{\small $H$};

\draw[black,fill] plot [smooth, tension=0.5] coordinates {(0.5,0) (0.8, 0.05) (1,0.1) (1.2,0.2) (1.5, 0.22) (2,0.45) (2.5, 0.5) (3,0.45) (3.5, 0.22) (3.8,0.2) (4,0.1) (4.2,0.05) (4.5,0) (4.2,-0.05) (4,-0.1) (3.8,-0.2) (3.5, -0.22) (3.5,-0.22) (3,-0.45) (2.5, -0.5) (2,-0.45) (1.5,-0.22) (1.2, -0.2) (1,-0.1) (0.8,-0.05) (0.5,0)};
    
\draw[->] (0,0) -- (5.5,0) ;
\draw[] (5.6,-0.6) node{ \small $x_1$};

            \end{scope}

\end{tikzpicture}

\caption{Example of rearrangements}
\label{picture:rearrangements}
\end{figure}

\subsection{Capacity}
\item For a compact set $K \subset \mathbb R^d$, its Newtonian capacity is defined as
$\text{Cap}(K) := \inf \int |\nabla f|^2 d\mu_d$, where the infimum is taken over all smooth functions $f$ such that $f \geq 1$ on $K$ and $f \to 0$ at infinity. Compact sets have non-zero finite capacity and we have the following scaling property: for any $a>0$
\begin{equation}\label{eq:scalingcapacity}
    \text{Cap}(a K) = a^{d-2} \text{Cap}(K) \, .
\end{equation}

The main property of rearrangements which we will use in this paper is the following:

\begin{theorem}[Symmetrizations decrease capacity, \cite{Sarvas}]\label{theorem:monotonicityrearrangements}
    Let $K$ be a compact set in $\mathbb R^d$ and $H$ be any $d-1$-dimensional hyperplane. We have both $\text{Cap}(Sc(K)) \leq \text{Cap}(\Omega)$ and $\text{Cap}(St_H(K)) \leq \text{Cap}(\Omega)$.
\end{theorem}

    \item Let $ D_{L,R,d}$ be the $d$-dimensional cylinder of length $L$ and radius $R$, i.e. $D_{L,R,d} := B_R^{d-1} \times (0,L)$.

Combining Proposition 3.4 in \cite{port-stone} with the scaling propertiy \eqref{eq:scalingcapacity} of Newtonian capacity, we obtain the following estimates for the capacity of thin cylinders:

\begin{lemma}[Capacity of thin cylinders, Proposition 3.4 in \cite{port-stone}]\label{le:capacityofcylinders}
    For any $d \geq 3$, there are constants $c_d$, $c_d'$  and $C_d$ with the following property: if $L > c_d' R$, then we have the following inequalities:

    \begin{align}
        c_d \frac{L}{\log (L/R)} &\leq \text{Cap}(D_{L,R,d}) \leq C_d \frac{L}{\log (L/R)} &&& \text{if} \quad d=3\, ,\\
        c_d LR^{d-3} &\leq \text{Cap}(D_{L,R,d}) \leq C_d LR^{d-3} &&& \text{if} \quad d \geq 4\, .
    \end{align}
\end{lemma}

\begin{remark}
    By contrast, $\mu_d(D_{L,R,d}) = c(d) LR^{d-1}$.
\end{remark}

We will use the following estimate on the capacity of the complement of nodal domains on small balls (see also \cites{ lieb, mazya-shubin}):
\begin{theorem}[Capacity estimates of nodal domains, Theorem 1.6 in \cite{geor-mukh} and Theorem 4.6 in \cite{cha-man}]\label{lemma:capacitycomplement}
There are constants $r_0(g)$, $\delta_0(g)$ and $C(g)$ with the following property:
Let $\Omega_\lambda$ be a nodal domain of $\phi_\lambda$ and $x$ be a point where $\phi_\lambda$ attains its maximum in $\Omega_\lambda$. Then, for any $\delta < \delta_0$, on the chart $U_x$ we have the following estimate:

\begin{equation}
    \text{Cap}\left(\Omega^c \cap B_{\frac{\delta}{\sqrt{\lambda}}} \right)\leq C \delta^2 \text{Cap}\left(B_{\frac{\delta}{\sqrt{\lambda}}}\right) \, .
\end{equation}
    
\end{theorem}

\end{itemize}

\subsection{Eigenfunctions and elliptic estimates}\label{subsection:tools}

We have the following fundamental bound on the doubling index for eigenfunctions:
\begin{theorem}[Theorem 4.2 (ii) in \cite{don-fef88}]\label{thm:df-growth}
    Let $(M, g)$ be a closed Riemannian manifold. Let $\phi_\lambda$ satisfy \linebreak $\Delta_g \phi_\lambda = \lambda \phi_\lambda$ in a chart $U_x$. There is a constant $C_{DF}(g)$ such that if $r < r_0/2$, for any $y \in B_r$ we have the following inequality:

    $$\mathcal{N}\left(\phi_\lambda, B_r(y)\right) \leq C_{DF}\sqrt{\lambda}\ .$$
   
\end{theorem}

The next result we will use is a direct corollary of the Remez inequality for solution of elliptic equations found in \cite{logu-mali-icm}. It will allow us to estimate the growth of eigenfunctions from arbitrary subsets of large measure in balls at sub-wavelength scales.

\begin{theorem}[Remez-type inequality for eigenfunctions in small scales, Corollary 2.5 in \cite{cha-man}]\label{cor:remez-eigen}
Let $(M, g)$ be a closed Riemannian manifold. Let $\phi_\lambda$ satisfy $-\Delta_g \phi_\lambda = \lambda \phi_\lambda$ in a chart $U_x$. There exist $\lambda_0(g)$ and $C_R(g)$ such that if $\lambda > \lambda_0(g)$ and $r<\lambda^{-1/2}$, then for any $y \in B_{r_0/2}$ and measurable set $E \subset B_r(y)$, we have the following inequality:

 \begin{equation}\label{eq:remez}
 \sup_{B_r(y)} |\phi_\lambda|\leq C_R\sup_E |\phi_\lambda| \left(C_R\frac{\mu_d(B_r(y))}{\mu_d(E)}\right)^{C_R N\left(\phi_\lambda, B_r(y)\right)+C_R}\ .
 \end{equation}

 \end{theorem}

Now, let 
$$L = a^{ij}\partial_i\partial_j+b^i\partial_i $$
be a uniformly elliptic operator of second order in the  unit ball.
Assume that 
$$\|a^{ij}\|_{C^2(B_1)}+\|b^i\|_{C^1(B_1)}\leq K$$
and that 
$$\forall \xi\in\Rb^d\quad a^{ij}\xi_i\xi_j \geq \kappa |\xi|^2$$
for some $K,\kappa>0$. We have the following estimates:
\begin{theorem}[Elliptic estimates {\cite{gil-tru}*{Theorems 3.1 and 6.2}}]
\label{thm:gradient}
Let $h$ satisfy $Lh=0$ in $B_1$. There exist  $C(K_1,\kappa)$ and $r_0(K,\kappa)$ such that for any $r < r_0$,

\begin{align}
    \sup\limits_{B_r} |h| &\leq \sup\limits_{\partial B_r} |h| \, ,\label{eq:maxprinciple1} \\
    |\nabla h(0)| &\leq \frac{C}{r}\sup_{B_r} |h|\ . \label{eq:boundsderivatives1}
\end{align}

\end{theorem}

If we apply estimates \eqref{eq:maxprinciple1} and \eqref{eq:boundsderivatives1} to the function $h_\lambda(x,y) = \phi_\lambda(x)e^{\sqrt{\lambda}y}$, the compactness of $M$ implies the following uniform estimates at sub-wavelength scales:

\begin{corollary}\label{cor:ellipticestimates}
    Let $(M,g)$ be a closed smooth manifold. If $\Delta_g \phi_\lambda = \lambda \phi_\lambda$ on $U_x$, then there exist constants $C_1(g), C_2(g)$ and $\lambda_0(g)$ such that if $\lambda > \lambda_0$, $r \leq C_1 \lambda^{-1/2}$ and $y \in B_\frac{r_0}{2}$ we have the following inequalities:

\begin{equation}\label{eq:maxprinciple}
    \sup\limits_{B_r(y)} |\phi_\lambda| \leq C_2 \sup\limits_{\partial B_r(y)} |\phi_\lambda| \, ,
\end{equation}

\begin{equation}\label{eq:boundsderivatives}
    |\nabla \phi_\lambda(y)|\leq \frac{C_2}{r}\sup_{B_r(y)} |\phi_\lambda|\ .
\end{equation}

\end{corollary}

\section{Estimating the capacity of certain sets using rearrangements}\label{sec:capacity}

We start by proving the following important lemma on the total capacity of a set which has a large intersection with many  concentric annuli of small width:

\begin{lemma}\label{lemma:additivityofcapacity}
There are constants $c(d)$ and $N_0(d)$ with the following property: Let  $K$ be  a compact subset of $\overline{B_1}$. Assume that $N > N_0$ and that there exists a sequence $2N+1 \leq n_1 < n_2 \ldots \leq n_{N} < 4N$ and balls $b_{j}$, $1 \leq j \leq N$ such that 

\begin{enumerate}
    \item $\text{diam} (b_{j}) = 1/(4N)$, 
    \item $b_j \subset B_{\frac{n_j}{4N}} \backslash B_{\frac{n_j-1}{4N}}$ and
    \item $\mu_d(K \cap b_j) \geq \frac{1}{2}\mu_d(b_j)$.
\end{enumerate}

Then,

\begin{align}
    \text{Cap}(K) &> \frac{c}{\log N} \quad &  \text{if} \quad  d=3,\\
    \text{Cap}(K) &> \frac{c}{N^{d-3}} \quad & \text{if} \quad d\geq 4.
\end{align}
\end{lemma}

\begin{proof}

Since the balls $b_j$ are disjoint, for some small constant $c(d)$ we have the following inequality:

\begin{equation}
    \mu_1 \bigg( \left\{ r \in (1/2,1) \quad | \quad \mu_{d-1}(K \cap \partial(B_r))  \geq  c N^{1-d}\right\}\bigg) \geq c.
\end{equation}

Let $Sc(K)$ be the spherical cap rearrangement of $K$, as defined in section \ref{subsection:rearrangements}. Since $N$ is large, we have the following inequality:

\begin{equation}\label{measureofbigsets}
    \mu_1\bigg({\left\{ r \in (1/2,1) \quad | \quad \mu_{d-1}(Sc(K) \cap \{x_1 = r \})  \geq  c N^{1-d}\right\}}\bigg) \geq c.
\end{equation}

Now, let $St_{H}(Sc(K))$ be the Steiner symmetrization of $Sc(\Omega)$ with respect to the hyperplane \linebreak $H:=\{x_1 = 1/2\}$ (see Figure \ref{picture:rearrangements}).

By inequality \eqref{measureofbigsets}, the set $St_{H}(Sc(K))$ contains a cylinder of length $c$ and radius $c/N$. 

By monotonicity of Newtonian capacity with respect to rearrangements (Theorem \ref{theorem:monotonicityrearrangements}), we have the following inequalities:

\begin{align}
    \text{Cap}(K) &\geq \text{Cap}(Sc(\Omega))\, ,\nonumber\\
    &\geq \text{Cap}(St_H(Sc(K))) \, ,\nonumber\\
    &\geq \text{Cap}(D_{c, \,c/N,\, d}) \, .
\end{align}
If we take $N_0(d)$ large enough and assume that $N > N_0$, applying Lemma \ref{le:capacityofcylinders} completes the proof of Lemma \ref{lemma:additivityofcapacity}.

\end{proof}

This leads to the following obvious corollary:

\begin{corollary}\label{cor:totalcapacityonannuli}
There are constants $C(d)$ and $\eta_0(d)$ such that the following holds:
    Assume that $r>0$, $\eta < \eta_0(d)$ and that for some compact set $K \subset \overline{B_r}$, $\text{Cap}(K) < \eta \, Cap(B_r)$. Choose $N=N(\eta)$ such that $N \leq e^{C(d)/\eta}$ (if $d=3$) or $N \leq C(d)\eta^{3-d}$ (if $d \geq 4$). Let us be given a collection of balls $b_j$, $j=2N+1 \ldots 4N$ of radius $r/(4N)$ such that $b_j \subset B_{\frac{rj}{4N}} \backslash B_{\frac{r(j-1)}{4N}}$.
    
    Then, there is $J \geq N$ and a sequence $2N+1 \leq j(1) < j(2) \ldots < j(J) \leq 4N$ such that \linebreak $\mu_d\left( K \cap b_{j(n)}\right) <  \mu_d(b_{j(n)})/2$.
    
\end{corollary}

\section{Growth of $\phi_\lambda$ in small balls if $\Omega_\lambda^c$ has small capacity}

We will now look specifically at eigenfunctions at sub-wavelength scales. For any $x \in M$, we assume that we are working in a chart $U_x \subset \mathbb R^d$ as defined in section \ref{subsection:notation}.

\begin{lemma}\label{lemma:inductive}
There are constants $\delta_0(g)$,  $\eta_0(d)$, $\lambda_0(g)$, $L(g)$ and $c(g)$  with the following property:
Let $\delta < \delta_0$, $\eta < \eta_0$, $\lambda > \lambda_0$ and $\phi_{\lambda}$ be a solution to $\Delta_{g} \phi_\lambda= \lambda \phi_\lambda$ on $U_{x}$.
Assume that $\Omega$ is open subset of  ${B_{\frac{\delta}{\sqrt{\lambda}}}}$ and that $\text{Cap}(\Omega^c \cap B_{\frac{\delta}{\sqrt{\lambda}}}) < \eta \,  Cap(B_{\frac{\delta}{\sqrt{\lambda}}})$.  Choose $N(\eta)$ as in Corollary \ref{cor:totalcapacityonannuli} and choose balls $b_j$ of radius $\frac{\delta}{4N\sqrt{\lambda}}$, $j=2N+1 \ldots 4N$, such that
$b_j \subset B_{\frac{\delta j}{4N\sqrt{\lambda}}} \backslash B_{\frac{\delta (j-1)}{4N\sqrt{\lambda}}}$ and 
\begin{equation}\label{eq:ballsmaxchoice}
    \sup\limits_{b_{j}} |\phi_{\lambda}| = \sup\limits_{\partial B_{\frac{\delta j}{4N\sqrt{\lambda}}}} |\phi_{\lambda}|\,.
\end{equation}

Now, choose $J \geq N$ and $j(n)$, $n=1 \ldots J$ as in Corollary \ref{cor:totalcapacityonannuli}. 
    Assume that $\sup\limits_\Omega |\phi_{\lambda}|\leq 1$ and that for some $1 \leq n \leq J-1$, $\sup\limits_{B_{\frac{j(n)}{4N}}} |\phi_{\lambda}|> L$. Then,  
    
    \begin{equation}
        \sup\limits_{B_{\frac{\delta j(n+1)}{4N\sqrt{\lambda}}}} |\phi_{\lambda}|>\sup\limits_{B_{\frac{\delta j(n)}{4N\sqrt{\lambda}}}} |\phi_{\lambda}|^{1+c}\,.
    \end{equation}
\end{lemma}

\begin{remark}
    This Lemma allows us to estimate the growth of $|\phi_\lambda|$ in terms of the intersection of $\Omega_\lambda^c$ with a sequence of concentric annuli, which will allow us to use Corollary \ref{cor:totalcapacityonannuli} in conjunction with Theorem \ref{lemma:capacitycomplement}. It is an improvement over lemmas 3.1 and 3.2 in \cite{cha-man}, in which the authors find conditions on $\delta$ and $N$ such that every ball $b(j)$ has small intersection with $\Omega_\lambda^c$. 
\end{remark}

\begin{proof}
Since $\lambda > \lambda_0$ and $\delta < \delta_0$,  we have by inequality \eqref{eq:maxprinciple} and equation \eqref{eq:ballsmaxchoice} that 

\begin{align}\label{eq:growthinbigballs}
    \sup\limits_{B_{\frac{\delta j(n)}{4N\sqrt{\lambda}}}} |\phi_{\lambda}| \leq C_1\sup\limits_{b_{j(n)}} |\phi_{\lambda}| \, .
\end{align}

We now introduce the following result on the growth of  $\phi_\lambda$:

\begin{lemma}[Lemma 3.2 in \cite{cha-man}]\label{lemma:growth}
    There exist constants $C_2(g)$ and $c_2(g)$ such that if $\sup\limits_{\Omega} |\phi_{\lambda}| \leq 1$, \linebreak $\sup\limits_{b_{j(n)}} |\phi_{\lambda}| > C_2$, and $\mu_d(b_{j(n)} \cap \Omega) \geq  \mu_d(b_{j(n)})/2$, then

\begin{equation}\label{eq:eqgrowthcharronmangoubi}
   \sup\limits_{2 b_{j(n)}} |\phi_{\lambda}| \geq \left(\sup\limits_{b_{j(n)}} |\phi_{\lambda}| \right)^{1+c_2} \, .
\end{equation}
\end{lemma}

For the convenience of the reader, here is a proof of Lemma \ref{lemma:growth}.

\begin{proof}
By Theorem \ref{cor:remez-eigen},

\begin{align*}
    \sup\limits_{b_{j(n)}}|\phi_\lambda| \leq C_R \sup_\Omega |\phi_\lambda| \left( \frac{C_R \mu_d(b_{j(n)})}{\mu_d(b_{j(n)} \cap \Omega)} \right)^{C_R \mathcal{N}(\phi_\lambda, b_{j(n)})+C_R} \, ,\\
    \leq C_R (2C_R)^{C_R \mathcal{N}(\phi_\lambda, b_{j(n)})+C_R} \, .
\end{align*}

Therefore, $\mathcal{N}(\phi_\lambda,b_{j(n)}) \geq c \log_2(\sup\limits_{b_{j(n)}}|\phi_\lambda| )-C$, which implies that $\mathcal{N}(\phi_\lambda,b_{j(n)}) \geq c_2 \log_2(\sup\limits_{b_{j(n)}}|\phi_\lambda| )$ if $\log_2(\sup\limits_{b_{j(n)}}|\phi_\lambda| ) $ is large enough and $c_2$ is taken small enough. This in turn implies that 

\begin{equation}
    \log_2 \frac{\sup\limits_{2b_{j(n)}}|\phi_\lambda| }{\sup\limits_{b_{j(n)}}|\phi_\lambda| } \geq c_2 \log_2 \sup\limits_{b_{j(n)}}|\phi_\lambda| \, ,
\end{equation}
which is equivalent to

\begin{equation}
\sup\limits_{2b_{j(n)}}|\phi_\lambda| \geq \sup\limits_{b_{j(n)}}|\phi_\lambda|^{1+c_2} \, .
\end{equation}

\end{proof}

Now,  fix $L(g) \geq C_2$ and $c(g)$ such that for any $l \geq L$ , 

\begin{equation}\label{eq:ineqconstants}
    l^{1+c} \leq (l/C_1)^{1+c_2} \, .
\end{equation}.

Assuming that $\sup\limits_{B_{\frac{\delta j(n)}{4N\sqrt{\lambda}}}} |\phi_{\lambda}|> L$, applying inequalities \eqref{eq:growthinbigballs}, \eqref{eq:eqgrowthcharronmangoubi} and \eqref{eq:ineqconstants} gives us the following:

\begin{align}
    \sup\limits_{B_{\frac{\delta j(n+1)}{4N\sqrt{\lambda}}}} |\phi_{\lambda}| \geq   &\, \, \, \left.\sup\limits_{2 b_{j(n)}} |\phi_{\lambda}|\right.  \, ,\nonumber\\
    \geq &\left(\sup\limits_{B_{\frac{\delta j(n)}{4N\sqrt{\lambda}}}} |\phi_{\lambda}| /C_1\right)^{1+c_2} \, , \nonumber\\
    \geq &\, \, \,\sup\limits_{B_{\frac{\delta j(n)}{4N\sqrt{\lambda}}}} |\phi_{\lambda}|^{1+c} \, ,
\end{align}
which completes the proof of Lemma \ref{lemma:inductive}.

\end{proof}

Applying Lemma \ref{lemma:inductive} recursively, we obtain the following:
\begin{corollary}\label{lemma:growthmaximal}
There are constants $\delta_0(g)$,  $\eta_0(d)$, $\lambda_0(g)$, $L(g)$ and $c(g)$  with the following property:
Let $\delta < \delta_0$, $\lambda > \lambda_0$, $\eta < \eta_0$ and $\phi_{\lambda}$ be a solution to $\Delta_{g} \phi= \lambda \phi$ on $U_{x}$.
Assume that $\Omega$ is an open subset of ${B_{\frac{\delta}{\sqrt{\lambda}}}}$ and that $\text{Cap}(\Omega^c \cap B_{\frac{\delta}{\sqrt{\lambda}}}) \leq \eta \, Cap(B_{\frac{\delta}{\sqrt{\lambda}}})$. If $\sup\limits_{B_{\frac{\delta}{2\sqrt{\lambda}}}} |\phi_{\lambda}| \geq L \sup\limits_\Omega|\phi_{\lambda}|$, then  

\begin{align}
    \mathcal{N}(B_{\frac{\delta}{2\sqrt{\lambda}}},\phi_{\lambda}) &\geq e^{e^{c/\eta}} \quad && \text{if}\quad d=3 \, ,\\
     \mathcal{N}(B_{\frac{\delta}{2\sqrt{\lambda}}},\phi_{\lambda}) &\geq e^{c \eta^{3-d}} \quad && \text{if} \quad d \geq 4 \, .
\end{align}

\end{corollary}

\begin{proof}
Setting $\widetilde{\phi_\lambda} := \frac{\phi_\lambda}{ \sup\limits_\Omega|\phi_{\lambda}|}$, we now fix $N = e^{C(d)/\eta}$ (if $d=3$) or $N= C(d) \eta^{3-d}$ (if $d \geq 4$), balls $b_j$, $2N+1 \leq j \leq 4N$, $J \geq N$ and a subsequence $j(n)$, $1 \leq n \leq J$, as in Lemma \ref{lemma:inductive}.

Trivially, $\sup\limits_{B_{\frac{\delta j(1)}{4N \sqrt{\lambda}}}} |\widetilde{\phi_\lambda}| \geq \sup\limits_{B_{\frac{\delta}{2\sqrt{\lambda}}}} |\widetilde{\phi_\lambda}|$. We can apply Lemma \ref{lemma:inductive} $J-1$ times to obtain that 

\begin{equation}
    \sup\limits_{B_{\frac{\delta}{\sqrt{\lambda}}}} |\widetilde{\phi_\lambda}| \geq L^{(1+c(g))^{N-1}}  \, .
\end{equation}

Therefore, 

\begin{align}\label{eq:lowerbounddoubling}
    \mathcal{N}(B_{\frac{\delta}{2\sqrt{\lambda}}},\phi_{\lambda}) &\geq \log_2 L^{{(1+c)}^{N-1}-1}\, ,\\
    &\geq e^{cN} \, ,
\end{align}
which completes the proof of Lemma \ref{lemma:growthmaximal}.
\end{proof}

\section{Growth of $|\phi_\lambda|$ in small balls around nodal maxima: completing the proof of Theorem \ref{theorem:maintheorem}}

Combining Corollary \ref{lemma:growthmaximal} with Theorem \ref{lemma:capacitycomplement}, we obtain the following upper bound on $|\phi_\lambda|$ on small balls around the origin on the chart $U_x$:

\begin{lemma}\label{lemma:estimatearoundmaximum}
Let $\Omega_\lambda$ be a nodal domain of $\phi_\lambda$ and $x$ be any point where $|\phi_\lambda|$ attains its maximum on $\Omega_\lambda$. Then, there are constants $\lambda_0(g)$, $c(g)$ and $L(g)$ such that if $\lambda > \lambda_0$ and if we choose $\delta(\lambda)$ as
\begin{align}
    \delta(\lambda):= \quad &c(g)(\log\log\lambda)^{-1/2}  \quad & if \quad d=3 \, ,\\
    \delta(\lambda):= \quad &c(g)(\log\lambda)^{\frac{3-d}{2}} \quad & if \quad d \geq 4 \, ,
\end{align}

then, in the chart $U_x$, $\sup\limits_{B_{\frac{\delta}{2 \sqrt{\lambda}}}}|\phi_\lambda| \leq L \phi_\lambda (0)$.

\end{lemma}

\begin{proof}
Take $\delta(\lambda) < \delta_0$ to be determined later, where $\delta_0$ is taken from Lemma \ref{lemma:growthmaximal}. 
 
 We now assume that on $U_x$, $\sup\limits_{B_{\frac{\delta}{2 \sqrt{\lambda}}}}|\phi_\lambda| > L \phi_\lambda (0)$.

 By Theorem \ref{lemma:capacitycomplement}, $\text{Cap}(B_{\frac{\delta}{2 \sqrt{\lambda}}} \cap \Omega_{\lambda}^c) \leq C(g)\delta^2 Cap(B_{\frac{\delta}{2 \sqrt{\lambda}}})$. There exists $\delta'(g)$ such that if $\delta < \delta'$, taking $\eta = C(g)\delta^2$ satisfies the assumptions of Corollary \ref{lemma:growthmaximal}. Therefore,  there exists $c(g)$ such that

\begin{align}\label{eq:growthestimates}
    \mathcal{N}(B_{\frac{\delta}{2 \sqrt{\lambda}}},\phi_\lambda) &\geq e^{e^{c \delta^{-2}}} &&  \text{if}\quad d=3\, ,\nonumber\\
 \mathcal{N}(B_{\frac{\delta}{2 \sqrt{\lambda}}},\phi_\lambda) &\geq e^{c \delta^{\frac{2}{3-d}}} && \text{if}\quad d \geq 4 \, .
 \end{align}

 Theorem \ref{thm:df-growth} ensures that

\begin{equation}\label{eq:DonFef}
    \mathcal{N}(B_{\frac{\delta}{2 \sqrt{\lambda}}},\phi_\lambda)\leq C_{DF}\sqrt{\lambda}.
\end{equation}

Therefore, if $\delta< \min(\delta_0, \delta')$ and $\sup\limits_{B_{\frac{\delta}{2 \sqrt{\lambda}}}}|\phi_\lambda| > L \phi_\lambda (0)$,  combining inequalities \eqref{eq:growthestimates} and \eqref{eq:DonFef} gives us that

\begin{align}
   \delta &\geq c (\log\log\lambda)^{-1/2} && \text{if} \quad  d=3 \, ,\\
   \delta &\geq c (\log\lambda)^{\frac{3-d}{2}} && \text{if} \quad  d \geq 4 \, .
\end{align}

which completes the proof of Lemma \ref{lemma:estimatearoundmaximum}.

\end{proof}

We are now ready to complete the proof of Theorem \ref{theorem:maintheorem}.

\begin{proof}[Proof of Theorem \ref{theorem:maintheorem}]
    Let $\Omega_\lambda$ be a nodal domain of $\phi_\lambda$ and $x$ be a point where $\phi_\lambda$ attains its maximum on $\Omega_\lambda$. Let $\lambda >\lambda_0(g)$ and $\delta(\lambda)$ be defined as in Lemma \ref{lemma:estimatearoundmaximum}.
     By Lemma \ref{lemma:estimatearoundmaximum}, there exists $L(g)$ such that on the chart $U_x$, $$\sup\limits_{B_{\frac{\delta}{2 \sqrt{\lambda}}}} |\phi_\lambda| \leq L\phi_\lambda(0) \, .$$

    Since $\lambda > \lambda_0$ implies that $\delta(\lambda)<\delta_0$, elliptic estimates (Corollary \ref{cor:ellipticestimates}) imply that

    \begin{align}
        \sup\limits_{B_{\frac{\delta}{4 \sqrt{\lambda}}}}|\nabla \phi_\lambda| \leq \frac{C L\sqrt{\lambda}\phi_\lambda(0)}{\delta} \, .
    \end{align}

Therefore, $\phi_\lambda$ does not change sign on $B_{\frac{c \delta}{\sqrt{\lambda}}}$. Since $M$ is smooth and closed, there is a geodesic ball of radius $\frac{c \delta}{\sqrt{\lambda}}$ around $x$ where $\phi_\lambda$ does not change sign, which completes the proof of Theorem \ref{theorem:maintheorem}.

    \end{proof}

\section{Proof of Theorem \ref{theorem:maintheoremshrinking}}\label{section:proofofshrinking}

Let $r_d$ be chosen such that $\lambda_1(B_{r_d}) = 1$. For a fixed $n$, we choose $A_{n,d}$ as $B_{r_d}$ with a finite collection of thin sets removed such that $\lambda_1(A_{n,d}) < 2$, $\text{Inrad}(A_{n,d}) \leq 1/n$ and $\partial A_{n,d}$ is smooth and connected\footnote{Note that the choice of $1/n$ is arbitrary, and can be replaced by any quantity that is $o_{n \to \infty}(1).$}. This is possible since segments have capacity zero in $\mathbb R^d$ if $d \geq 3$ (see \cite{rau-tay}), capacity is subadditive and we have convergence of Dirichlet eigenvalues for sequences of increasing domains (Proposition 2.4 in \cite{ber-col}). We then set $U_{n,d}:= \sqrt{\lambda_1(A_{n,d})} A_{n,d}$. By construction, $\lambda_1(U_{n,d})=1$, $U_{n,d} \subset B_{2r_d}$ and $\text{Inrad}(U_{n,d})< \sqrt{2}/n$.

We now identify $\mathbb T^d$ with $[-\pi, \pi]^d$. By theorem 7.1 of \cite{enc-persal1}, for any $\epsilon >0$, there exists $k(n,d,\epsilon)$, an eigenvalue $\lambda_{k}$ (which we periodically extend to $\mathbb R^d$), a toral eigenfunction $\phi_{\lambda_{k}}$ with a nodal domain $\Omega_{\lambda_{k}}$ and a diffeomorphism $\psi_k: \mathbb R^d \to \mathbb R^d$ such that $||\psi_k- \text{Id}||_{C^1} \leq \epsilon$ and$\Omega_{\lambda_{k}} = {\lambda_{k}}^{-1/2} \psi(U_{n,d})$.

Fix $\epsilon(n,d)$ small enough such that, for some $k(n,d, \epsilon(n,d))$ large enough, $\text{Inrad}(\Omega_{\lambda_{k}}) \leq \frac{2}{n\sqrt{\lambda_{k}}}$. A simple diagonalization argument completes the proof of Theorem \ref{theorem:maintheoremshrinking}.

\vspace{3ex}

\end{document}